\documentclass[11pt]{amsproc}

\usepackage{amsmath, amsthm, amssymb, amsfonts, amscd}

\def\AA{{\mathbb A}}

\def\FF{{\mathbb F}}

\def\NN{{\mathbb N}}

\def\QQ{{\mathbb Q}}
\def\PP{{\mathbb P}}
\def\QQ{{\mathbb Q}}
\def\RR{{\mathbb R}}

\def\pfr{{\mathfrak p}}

\def\afr{{\mathfrak a}}

\def\0{{\mathbf 0}}
\def\1{{\mathbf 1}}

\def\Xbf{{\mathbf X}}

\def\Mcal{{\mathcal M}}

\def\triv{\mathrm{triv}}

\def\Spec{\mathrm{Spec}}

\def\supp{\mathrm{supp}}

\def\sup{\mathrm{sup}}

\theoremstyle{plain}

\newtheorem*{PRT}{Poincar\'e Recurrence Theorem}

\newtheorem{thm}{Theorem}

\newtheorem{prop}[thm]{Proposition}

\theoremstyle{definition}

\newtheorem{rem}{Remark}
\newtheorem{ex}{Example}

\title[]{On the distribution of orbits in affine varieties}

\author{Clayton Petsche}

\address{Clayton Petsche; Department of Mathematics; Oregon State University; Corvallis OR 97331 U.S.A.}

\email{petschec@math.oregonstate.edu}

\date{First version January 13 2014, revised February 24 2014}

\begin{document}

\begin{abstract}
Given an affine variety $X$, a morphism $\phi:X\to X$, a point $\alpha\in X$, and a Zariski closed subset $V$ of $X$, we show that the forward $\phi$-orbit of $\alpha$ meets $V$ in at most finitely many infinite arithmetic progressions, and the remaining points lie in a set of Banach density zero.  This may be viewed as a weak asymptotic version of the Dynamical Mordell-Lang Conjecture for affine varieties.  The results hold in arbitrary characteristic, and the proof uses methods of ergodic theory applied to compact Berkovich spaces.
\end{abstract}

\maketitle

\section{Introduction}

Let $X$ be an affine variety and let $\phi:X\to X$ be a morphism.  Let $\NN$ denote the set of nonnegative integers, and for each $n\in\NN$ write $\phi^n=\phi\circ\dots\circ\phi$ for the $n$-fold composition of $\phi$ with itself.  Given a point $\alpha\in X$, it is a fundamental question of algebraic dynamics to describe the distribution in $X$ of the orbit
\begin{equation}\label{VOrbit}
\alpha, \phi(\alpha), \phi^2(\alpha), \phi^3(\alpha), \dots
\end{equation}
of $\alpha$ with respect to iteration of $\phi$.  

An informal governing philosophy states that the orbit $(\ref{VOrbit})$ should distribute as generically as possible, except for the possible existence of trivial obstructions.  One way to make this idea more precise is to consider an aribtrary Zariski-closed subset $V$ of $X$, and to ask how often along the orbit $(\ref{VOrbit})$ does it occur that $\phi^n(\alpha)\in V$?  It may happen that $\phi^n(\alpha)\in V$ for many integers $n$, but for the trivial reason that $\phi^b(\alpha)\in V_0\subseteq V$ for some integer $b\geq0$ and some $\phi$-periodic subvariety $V_0$ of $V$; we say $V_0$ is {\em $\phi$-periodic} if there exists an integer $a\geq1$ such that $\phi^a(V_0)\subseteq V_0$.  In this case, we must have $\phi^n(\alpha)\in V$ for all $n$ along the the infinite arithmetic progression $a\NN+b=\{a\ell+b\mid \ell\in\NN\}$.  But excluding this possibility, it should be a rare occurence for $V$ to contain $\phi^n(\alpha)$. 

\begin{thm}\label{MainThmIntro}
Let $X$ be an affine variety, let $\phi: X\to X$ be a morphism, let $\alpha\in X$ be a point, and let $V$ be a Zariski-closed subset of $X$.  If the set $\{n\in\NN\mid\phi^n(\alpha)\in V\}$ contains no infinite arithmetic progressions, then it has Banach density zero.
\end{thm}

\begin{thm}\label{MainCorIntro}
Let $X$ be an affine variety, let $\phi: X\to X$ be a morphism, let $\alpha\in X$ be a point, and let $V$ be a Zariski-closed subset of $X$.  Then $\{n\in\NN\mid\phi^n(\alpha)\in V\}=A\cup B$, where $A$ is a (possibly empty) finite union of infinite arithmetic progressions, and $B$ is a set of Banach density zero.
\end{thm}

To specify some of our terminology, by an affine variety we mean $\Spec(A)$ for a finitely generated $k$-algebra $A$ over an arbitrary field $k$.  The point $\alpha$ in the statements of Theorems~\ref{MainThmIntro} and \ref{MainCorIntro} is not assumed to be a closed point.

Given a subset $S\subseteq \NN$, the {\em Banach density} of $S$ is defined to be 
\begin{equation*}
d(S)=\limsup_{|I|\to+\infty}\frac{|S\cap I|}{|I|}
\end{equation*}
the limit supremum taken over all intervals $I$ in $\NN$ with $|I|\to+\infty$.  If $S$ has Banach density zero, then it necessarily has {\em ordinary} density zero in the sense that $\frac{1}{n}|S\cap[0,n-1]|\to0$ as $n\to+\infty$.  Simple examples\footnote{$S=\cup_{n\geq1}\{n^3,n^3+1,\dots,n^3+n\}$} show that the converse is false in general, and so one may view a set having Banach density zero as ``thinner'' than an arbitrary set of ordinary density zero.

Since infinite arithmetic progressions have positive Banach density, Theorem~\ref{MainThmIntro} is formally weaker than Theorem~\ref{MainCorIntro}.  However, Theorem~\ref{MainThmIntro} is the main result of this paper and is proved independently.  Theorem~\ref{MainCorIntro} is proved using Theorem~\ref{MainThmIntro} and a bootstrapping argument via induction on the dimension of $V$.  The main idea behind this argument is not new; similar arguments are used, for example, in Ghioca-Tucker \cite{MR2521481} and Bell-Ghioca-Tucker \cite{MR2766180}. 

Our proof of Theorem~\ref{MainThmIntro} uses methods from topological dynamics and ergodic theory, and our primary tool is a strong topological version of the Poincar\'e Recurrence Theorem due to Furstenberg \cite{MR628658}.  In order to apply this approach, we lift the dynamical system $\phi:X\to X$ to a compact Berkovich space $\Xbf$ via a surjective, anti-continuous reduction map $\pi:\Xbf\to X$.  If the set $\{n\in\NN\mid\phi^n(\alpha)\in V\}$ contains many points, then we can produce a measure-preserving dynamical system on $\Xbf$ which positively charges $\pi^{-1}(V)$; the existence of a suitable recurrent point associated to this dynamical system leads to the existence of the desired periodic subvariety of $V$.

While preparing this article we learned that two groups of researchers have independently proved Theorem~\ref{MainCorIntro} for arbitrary Zariski spaces $X$, which includes all algebraic varieties.  First, Gignac uses a study of invariant measures and ergodic theory directly on the Zariski topology of $X$ and does not pass to Berkovich spaces.  The necessary theory of Borel measures on Zariski spaces is built in Gignac \cite{GignacPaper}, and in the same paper he derives a new measure-theoretic proof of a dynamical result of Favre (\cite{FavreThesis} Thm. 2.5.8).  The generalization of Theorem~\ref{MainCorIntro} uses this result of Favre, and is alluded to in Gignac's thesis \cite{GignacThesis}, but the proof has not been published.  Independently, Bell-Ghioca-Tucker \cite{BellGhiocaTucker2014} have given a proof of Theorem~\ref{MainCorIntro} for Zariski spaces, using an elementary approach which leads to interesting quantitative results.

In characteristic zero it has been conjectured that Theorems~\ref{MainThmIntro} and \ref{MainCorIntro} are true (for arbitrary quasiprojective varieties) in the strengthened form in which the Banach density zero statements are improved to finiteness statements.  This is known as the Dynamical Mordell-Lang (DML) Conjecture, and was stated at this level of generality by Ghioca-Tucker \cite{MR2521481}, following less general statements by Denis \cite{MR1259107} and Bell \cite{MR2225492}.  While the DML Conjecture remains open at present, a great deal of progress has been made in special cases; a non-exhaustive list would include the articles of Denis, Bell, and Ghioca-Tucker already cited, as well as Bell-Ghioca-Tucker \cite{MR2766180}, and Ghioca-Tucker-Zieve \cite{MR2367026}.  
 
Many proofs of special cases of the standard characteristic zero DML Conjecture use variations on the Skolem-Mahler-Lech method of $p$-adic analysis, often combined with auxiliary results from arithmetic geometry and number theory.  However, the Skolem-Mahler-Lech method breaks down in characteristic $p$ due to the lack of a suitable $p$-adic logarithm.  This is illustrated by the following example which was pointed out to us by Tom Tucker.

\begin{ex}\label{CharpEx}
Let $k=\FF_p(t)$ denote the field of rational functions over the finite field $\FF_p$, define the map $\phi:\AA^2\to\AA^2$ over $k$ by $\phi(x,y)=(tx,(1-t)y)$, and consider the point $\alpha=(1,1)$ and the line $V=\{x+y=1\}$ in $\AA^2$.  Then $\phi^n(\alpha)=(t^n,(1-t)^n)$ and $\{n\in\NN\mid\phi^n(\alpha)\in V\}=\{p^\ell\mid \ell\geq0\}$, a set which is not a union of finitely many infinite arithmetic progressions with finitely many integers.
\end{ex}

Example~\ref{CharpEx} shows that in characteristic $p$, Theorems~\ref{MainThmIntro} and \ref{MainCorIntro} cannot be proved in the strengthened form in which the Banach density zero statements are improved to finiteness statements.  In other words, no method which is insensitive to the characteristic --- and our approach to the proofs of Theorems~\ref{MainThmIntro} and \ref{MainCorIntro} is one such method --- could possibly produce a full proof of the standard characteristic zero DML Conjecture. 

Interestingly, Denis \cite{MR1259107} also proved a Banach density zero result in the setting of an automorphism of $\PP^n$ in characteristic $p$.  In place of the Skolem-Mahler-Lech method, he used a result of B\'ezivin \cite{MR919425} on recurrence sequences in characteristic $p$.  B\'ezivin's theorem, in turn, relies on the well-known theorem of Szemer\'edi on arithmetic progressions occurring in subsets of $\NN$ with positive Banach density.  In view of Furstenberg's \cite{MR0498471} ergodic-theoretic proof of Szemer\'edi's theorem, there has been, at least indirectly, a connection between ergodic theory and the Dynamical Mordell-Lang problem for some time.  

I thank William Gignac and Tom Tucker for helpful conversations.  This research was supported in part by NSF grant DMS-0901147.


\section{A few facts from topological dynamics and ergodic theory}\label{TopSect}

In this section $S$ denotes a compact metric space.  Given a continuous map $T: S\to S$ and a unit Borel measure $\mu$ on $S$, one says that $\mu$ is {\em $T$-invariant} if $\mu(T^{-1}(E))=\mu(E)$ for all Borel subsets $E$ of $S$.  Alternatively, a standard argument shows that $\mu$ is {\em $T$-invariant} if and only if $\int(f\circ T)d\mu=\int fd\mu$ for all continuous functions $f: S \to\RR$ (\cite{MR648108} Thm. 6.8).  

A point $\alpha\in S $ is said to be {\em $T$-recurrent} if $\alpha$ is a limit point of the sequence $\{T^n(\alpha)\}$.  There are several categories of results in the literature whose purpose is to show the existence, or even plentitude, of recurrent points in various settings.  The main dynamical tool of this article is the following strong topological version of the Poincar\'e Recurrence Theorem due to Furstenberg; for the proof see \cite{MR628658} Thm. 1.1 and the three subsequent paragraphs.

\begin{PRT}
Let $ S $ be a separable compact metric space, let $T: S \to S $ be a continuous map, and let $\mu$ be a $T$-invariant unit Borel measure on $ S $.  Then $\mu$-almost all points of $ S $ are $T$-recurrent.  
\end{PRT}

A sequence $\{\mu_m\}$ of unit Borel measures on $S$ is said to {\em converge weakly} to a unit Borel measure $\mu$ on $S$ if $\int fd\mu_m\to\int fd\mu$ for all continuous functions $f: S \to\RR$. 

\begin{rem}\label{PortRemark}
The Portmanteau theorem gives several criteria for weak convergence of unit Borel measures on metric spaces.  We will use one such result, namely that $\{\mu_m\}$ converges weakly to $\mu$ if and only if $\liminf\mu_m(U)\geq\mu(U)$ for all open subsets $U$ of $S$ (\cite{MR1700749} Thm. 2.1).
\end{rem}


\section{The Berkovich spectrum of a trivially-valued $k$-algebra}\label{BerkSpectSect}

Let $k$ be a field.  In this section we summarize the necessary facts about the Berkovich spectrum of a trivially-valued $k$-algebra.  Most of the facts stated in this section are due to Berkovich himself in the more general setting of arbitrary Banach rings; other facts presented here seem to be well-known, but we supply proofs for lack of suitable references.

Let $A$ be a $k$-algebra.  A {\em bounded multiplicative seminorm} on $A$ is a function $[\cdot]:A\to \RR$ satisfying
\begin{itemize}
\item[(i)] $0\leq[f]\leq1$ for all $f\in A$;
\item[(ii)] $[\cdot]$ restricts to the trivial absolute value on $k$;
\item[(iii)] $[f+g]\leq\max\{[f],[g]\}$ for all $f,g\in A$;
\item[(iv)] $[fg]=[f][g]$ for all $f,g\in A$.
\end{itemize}

\begin{rem}
This definition is equivalent to Berkovich's definition \cite{MR1070709} of a bounded multiplicative seminorm when $A$ is viewed as a Banach ring with respect to the the trivial norm $[\cdot]_0$.  In place of our conditions (i) and (ii), Berkovich requires only that $[1]=1$ and that $0\leq[f]\leq C[f]_0$ for some $C>0$.  Along with condition (iv) this implies that $[f]^n=[f^n]\leq C$ for each nonzero $f\in A$, whereby $[f]\leq1$ upon taking $n\to+\infty$.  $0\leq[0]\leq C[0]_0=0$, implying $[0]=0$, and if $a\in k^\times$ then $1=[1]=[aa^{-1}]=[a][a^{-1}]$ with both $[a]\leq1$ and $[a^{-1}]\leq1$, implying $[a]=1$.  Finally, Berkovich requires only the usual triangle inequality $[f+g]\leq [f]+[g]$, but the binomial theorem and the fact that $[\cdot]$ is trivially valued on $k$ implies that $[f+g]^n=[(f+g)^n]\leq \sum_{0\leq\ell\leq n}[f]^k[g]^{n-k}\leq(n+1)\max\{[f],[g]\}^n$.  Taking $n$-th roots and letting $n\to+\infty$ gives $[f+g]\leq\max\{[f],[g]\}$.
\end{rem}

The {\em Berkovich spectrum} of a $k$-algebra $A$ is defined to be the set $\Mcal(A)$ of all bounded multiplicative seminorms on $A$.  As a notational convention it is standard to refer to a point $\zeta\in\Mcal(A)$, and to denote by $[\cdot]_\zeta$ its corresponding seminorm on $A$.  The topology on $\Mcal(A)$ is defined to be the coarsest topology with respect to which the real-valued functions $\zeta\mapsto[f]_\zeta$ are continuous for all $f\in A$.  Equivalently, for each $f\in A$ and each pair $s<t$ of real numbers, define a subset of $\Mcal(A)$ by 
\begin{equation}\label{BerkBase}
U_{s,t}(f) = \{\zeta\in\Mcal(A) \mid s<[f]_\zeta<t\}.
\end{equation}
Then the collection of all nonempty finite intersections of sets of the form $U_{s,t}(f)$ forms a base of open sets for the topology on $\Mcal(A)$.

\begin{prop}\label{CompProp}
Let $A$ be a $k$-algebra.  Then $\Mcal(A)$ is a nonempty compact Hausdorff space.  If $A$ is countable then $\Mcal(A)$ is metrizable and separable.  
\end{prop}
\begin{proof}
For the proofs that $\Mcal(A)$ is nonempty, compact, and Hausdorff, see \cite{MR1070709} Thm 1.2.1.  If $A$ is countable, then the collection of all nonempty finite intersections of sets of the form $U_{s,t}(f)$, for $s,t\in\QQ$ and $f\in A$, forms a countable base of open sets for the topology on $\Mcal(A)$; in other words, $\Mcal(A)$ is second countable.  Since $\Mcal(A)$ is a compact Hausdorff space, it is normal (\cite{MR1681462} Thm. 4.25), and the Urysohn Metrization Theorem (\cite{MR1681462} Thm. 4.58) states that a space is metrizable provided it is normal and second countable.  Finally, a second countable space is always separable (\cite{MR1681462} Thm. 4.5).
\end{proof}

The Berkovich spectrum $\Mcal(A)$ of a $k$-algebra $A$ comes equipped with the {\em reduction map}
\begin{equation}\label{ReductionDef}
\pi:\Mcal(A)\to\Spec(A) \hskip1cm \pi(\zeta)=\{f\in A\mid[f]_\zeta<1\}.
\end{equation}
That $\{f\in A\mid[f]_\zeta<1\}$ is a prime ideal of $A$ follows easily from the properties (i), (iii), and (iv) of the definition of a bounded multiplicative seminorm.

In the opposite direction, each prime ideal $\pfr$ of $A$ gives rise to a bounded multiplicative seminorm obtained from the quotient map $A\to A/\pfr$ followed by the trivial norm on the integral domain $A/\pfr$.  We call
\begin{equation}\label{TrivialDef}
\triv:\Spec(A)\to\Mcal(A) \hskip1cm [f]_{\triv(\pfr)}=\begin{cases} 0 & \text{ if } f\in \pfr \\ 1 & \text{ if } f\notin \pfr.  \end{cases}
\end{equation}
the {\em trivial map}.

\begin{rem}\label{RedTrivRemark}
It follows at once from the definitions that the trivial map is a one-sided inverse of the reduction map, in the sense that $\pi\circ\triv$ is equal to the identity map on $\Spec(A)$.  In particular, $\triv$ is injective and $\pi$ is surjective.  On the other hand, for each $\zeta\in\Mcal(A)$, we have $\triv(\pi(\zeta))=\zeta$ if and only if the seminorm $[\cdot]_\zeta$ only takes the values $0$ and $1$.
\end{rem}

\begin{prop}
Let $A$ be a $k$-algebra.  
\begin{itemize}
	\item[{\sc (A)}]  (Berkovich \cite{MR1070709} Cor. 2.4.2)  If $A$ is Noetherian, then the reduction map $\pi:\Mcal(A)\to\Spec(A)$ is anti-continuous in the sense that the inverse image of every closed (resp. open) subset of $\Spec(A)$ is open (resp. closed) in $\Mcal(A)$.
	\item[{\sc (B)}]  The trivial map $\triv:\Spec(A)\to\Mcal(A)$ is closed. 
\end{itemize}  
\end{prop}
\begin{proof}
Consider an arbitrary  closed subset $V(\afr)=\{\pfr\mid\afr\subseteq\pfr\}$ of $\Spec(A)$, where $\afr$ is an ideal of $A$.  If $A$ is Noetherian then $\afr=(f_1,\dots,f_n)$ for some finite collection $f_1,\dots, f_n\in A$.  It follows that
\begin{equation*}
\begin{split}
\pi^{-1}(V(\afr)) & = \{\zeta\in\Mcal(A)\mid \pi(\zeta)\in V(\afr)\} \\
 & = \{\zeta\in\Mcal(A)\mid \afr\subseteq\{f\in A\mid[f]_\zeta<1\}\} \\
 & = \cap_{f\in\afr}\{\zeta\in\Mcal(A)\mid [f]_\zeta<1\} \\
 & = \cap_{1\leq j\leq n}\{\zeta\in\Mcal(A)\mid [f_j]_\zeta<1\}.
\end{split} 
\end{equation*}
Each set $\{\zeta\in\Mcal(A)\mid [f_j]_\zeta<1\}$ is open by the definition of the topology on $\Mcal(A)$, and therefore $\pi^{-1}(V(\afr))$ is open.

To show that the trivial map is closed, let $\langle\zeta_i\rangle$ be a net in $\triv(V(\afr))$ such that $\zeta_i\to\zeta$ in $\Mcal(A)$, thus for each index $i$ we may write $\zeta_i=\triv(\pfr_i)$ for $\pfr_i\in V(\afr)$.  Since for each $f\in A$ we have $[f]_{\zeta_i}\to[f]_\zeta$, and since each seminorm $[\cdot]_{\zeta_i}$ only takes the values $0$ and $1$, the same is true of $[\cdot]_{\zeta}$; in particular, by Remark~\ref{RedTrivRemark} it follows that $\zeta=\triv(\pfr)$ where $\pfr=\pi(\zeta)$.  If $f\in\afr$, then $f\in\pfr_i$ for all $i$, and so $[f]_{\zeta_i}=0$ for all $i$, implying that $[f]_\zeta=0$ and thus $f\in\pfr$.  We have shown that $\afr\subseteq \pfr$, so $\pfr\in V(\afr)$, and consequently $\zeta=\triv(\pfr)\in \triv(V(\afr))$.
\end{proof}

\begin{prop}\label{ContLift}
Let $A$ and $B$ be $k$-algebras and let $\phi:\Spec(A)\to\Spec(B)$ be a morphism of affine $k$-schemes.  There exists a continuous map $T_\phi:\Mcal(A)\to\Mcal(B)$ such that $\pi_B\circ T_\phi=\phi\circ\pi_A$ and $T_\phi\circ\triv_A=\triv_B\circ\phi$.
\begin{equation*}
\begin{CD}
\Mcal(A)  @> T_\phi >>  \Mcal(B) \\ 
@V \pi_{A} VV                                    @VV \pi_{B} V \\ 
\Spec(A)           @> \phi >>      \Spec(B)
\end{CD} \hskip1cm
\begin{CD}
\Mcal(A)  @> T_\phi >>  \Mcal(B) \\ 
@A \triv_{A} AA                                    @AA \triv_{B} A \\ 
\Spec(A)           @> \phi >>      \Spec(B)
\end{CD} 
\end{equation*}
\end{prop}
\begin{proof}
Denote by $\Phi:B\to A$ the $k$-algebra homomorphism obtained from viewing $A$ and $B$ as the rings of global sections on $\Spec(A)$ and $\Spec(B)$, respectively; thus $\phi(\pfr)=\Phi^{-1}(\pfr)$ for each $\pfr\in\Spec(A)$.  For each $\zeta\in\Mcal(A)$, define $[\cdot]_{T_\phi(\zeta)}:B\to\RR$ by $[f]_{T_\phi(\zeta)}=[\Phi(f)]_\zeta$.  Routine calculations verify that $[\cdot]_{T_\phi(\zeta)}$ is a bounded multiplicative seminorm on $B$, and we obtain a map $T_\phi:\Mcal(A)\to\Mcal(B)$.  Given $\zeta\in\Mcal(A)$, we have
\begin{equation*}
\begin{split}
\pi_B(T_\phi(\zeta)) & = \{g\in B\mid[g]_{T_\phi(\zeta)}<1\} \\
	& = \{g\in B\mid[\Phi(g)]_{\zeta}<1\} \\
	& = \Phi^{-1}(\{f\in A\mid[f]_{\zeta}<1\}) \\
	& = \phi(\pi_A(\zeta)).
\end{split} 
\end{equation*}
Given $\pfr\in\Spec(A)$, both of the seminorms $[\cdot]_{\triv_B(\phi(\pfr))}$ and $[\Phi(\cdot)]_{\triv_A(\pfr)}$ on $B$ only take the values $0$ and $1$.  Therefore for each $g\in B$ we have
\begin{equation*}
\begin{split}
[g]_{\triv_B\circ\phi(\pfr)} = 0 & \Leftrightarrow g\in\phi(\pfr)=\Phi^{-1}(\pfr) \\
	& \Leftrightarrow \Phi(g)\in \pfr  \\
	& \Leftrightarrow [\Phi(g)]_{\triv_A(\pfr)} = 0  \\
	& \Leftrightarrow [g]_{T_\phi(\triv_A(\pfr))} = 0,
\end{split} 
\end{equation*}
and we conclude that $\triv_B(\phi(\pfr))=T_\phi(\triv_A(\pfr))$.

To show that $T_\phi$ is continuous, it suffices to show that $T_\phi^{-1}(U_{s,t}(f))$ is an open subset of $\Mcal(A)$ for each $s,t\in\RR$ and $f\in B$, because sets of the form $U_{s,t}(f)$ (defined in $(\ref{BerkBase})$) generate the topology on $\Mcal(B)$.  This follows from the easily checked identity $T_\phi^{-1}(U_{s,t}(f))=U_{s,t}(\Phi(f))$.
\end{proof}


\section{The Proofs of Theorems \ref{MainThmIntro} and \ref{MainCorIntro}}

\begin{rem}\label{CountableRemark}
In order to prove Theorem~\ref{MainThmIntro} we may assume without loss of generality that the base field $k$ is either finite or countably infinite.  For suppose that $k$ is arbitrary and let $A=k[t_1,\dots, t_r]/I$ be a finitely generated $k$-algebra, where  $I$ is an ideal of $k[t_1,\dots, t_r]$, let $\phi:\Spec(A)\to \Spec(A)$ be a morphism, let $\alpha\in\Spec(A)$, and let $V=V(\afr)$ be an arbitrary Zariski-closed subset of $\Spec(A)$, where $\afr$ is an ideal of $A$.  Let $\Phi:A\to A$ be the $k$-algebra homomorphism obtained from viewing $A$ as the ring of global sections on $\Spec(A)$; thus $\phi(\pfr)=\Phi^{-1}(\pfr)$ for each $\pfr\in\Spec(A)$.  Let $k_0$ be a subfield of $k$ which, on the one hand is finitely generated over the prime subfield of $k$, but which on the other hand is large enough so that all coefficients of $\Phi$ and all generators of $I$, $\alpha$ (viewed as a prime ideal of $A$), and $\afr$ are elements of the image of $k_0[t_1,\dots,t_r]$ under the quotient map $k[t_1,\dots,t_r]\to A$.  Letting $I_0=I\cap k_0[t_1,\dots, t_r]$, we may view $A_0=k_0[t_1,\dots, t_r]/I_0$ as a $k_0$-subalgebra of $A$.  Thus $\Phi:A\to A$ restricts to a map $\Phi_0:A_0\to A_0$ and we obtain a morphism $\phi_0:\Spec(A_0)\to \Spec(A_0)$ defined by $\phi_0(\pfr)=\Phi_0^{-1}(\pfr)$.  Define $\alpha_0=\alpha\cap A_0$ and $\afr_0=\afr\cap A_0$.  A straightforward calculation (using the fact that inverse image commutes with intersection) shows that $\phi_0(\alpha_0)\in V(\afr_0)$ if and only if $\phi(\alpha)\in V(\afr)$, and more generally that 
\begin{equation*}
\{n\in\NN\mid\phi_0^n(\alpha_0)\in V(\afr_0)\}=\{n\in\NN\mid\phi^n(\alpha)\in V(\afr)\}.
\end{equation*}
\end{rem}

\begin{proof}[Proof of Theorem~\ref{MainThmIntro}]  As discussed in Remark~\ref{CountableRemark}, we may assume without loss of generality that $X=\Spec(A)$ for a finitely generated $k$-algebra $A$ over a countable field $k$.  Letting $\Xbf=\Mcal(A)$, it follows from Proposition~\ref{CompProp} that $\Xbf$ is a separable metrizable space.  As constructed in $\S$~\ref{BerkSpectSect} we denote by $\pi:\Xbf\to X$ the (anti-continuous) reduction map, and by $\triv:X\to \Xbf$ the (closed) trivial map.

Given a morphism $\phi: X\to X$, a point $\alpha\in X$, and a Zariski-closed subset $V$ of $X$, assume that the set $\{n\in\NN\mid\phi^n(\alpha)\in V\}$ has positive Banach density; we must show that this set contains an infinite arithmetic progression.  Denote by $T=T_\phi:\Xbf\to\Xbf$ the continuous lift of $\phi:X\to X$ constructed in Proposition~\ref{ContLift}.

The positive Banach density hypothesis implies that there exists a sequence $\{I_m\}$ of intervals in $\NN$ with $|I_m|\to+\infty$ and 
\begin{equation}\label{Epsilon}
\lim_{m\to+\infty}\frac{|\{n\in\NN\mid \phi^n(\alpha)\in V\}\cap I_m|}{|I_m|}>0.
\end{equation}
Define a sequence $\{\mu_m\}$ of unit Borel measures on $\Xbf$ by 
\begin{equation*}
\mu_m=\frac{1}{|I_m|}\sum_{n\in I_m}\delta_{T^n(\triv(\alpha))},
\end{equation*}
where $\delta_x$ denotes the Dirac measure supported at the point $x\in\Xbf$.  By Prokhorov's theorem (\cite{MR1700749} Thm 5.1), passing to a subsequence we may assume without loss of generality that the sequence $\{\mu_m\}$ converges weakly to a unit Borel measure $\mu$ on $\Xbf$.

\medskip

{\em Claim 1:  $\mu$ is $T$-invariant.}  

\medskip

{\em Claim 2: $\mu(\pi^{-1}(V))>0$.}  

\medskip

Postponing the proofs of the two claims, we will now complete the proof of Theorem~\ref{MainThmIntro}.  

The Poincar\'e Recurrence Theorem (see $\S$~\ref{TopSect}) implies that $\mu$-almost all points of $\Xbf$ are $T$-recurrent, and since $\mu(\pi^{-1}(V))>0$, there exists a $T$-recurrent point $\zeta\in\pi^{-1}(V)\cap\supp(\mu)$.  Further, $\zeta$ must be a limit point of the forward $T$-orbit of $\triv(\alpha)$.  For if $U$ is an open subset of $\Xbf$ containing $\zeta$, then $\liminf\mu_m(U)\geq\mu(U)>0$ by Remark~\ref{PortRemark} and the fact that $\zeta\in\supp(\mu)$; it follows that $T^n(\triv(\alpha))\in U$ for some (in fact infinitely many) $n\geq1$.

Since $\zeta$ is $T$-recurrent, the open neighborhood $\pi^{-1}(\overline{\{\pi(\zeta)\}})$ of $\zeta$ in $\Xbf$ contains $T^a(\zeta)$ for some $a\geq1$.  In other words 
\begin{equation*}
\phi^a(\pi(\zeta))=\pi(T^a(\zeta))\in \overline{\{\pi(\zeta)\}},
\end{equation*}
and in particluar this implies that 
\begin{equation*}
\phi^a(\overline{\{\pi(\zeta)\}})\subseteq\overline{\{\phi^a(\pi(\zeta))\}}\subseteq\overline{\{\pi(\zeta)\}}.
\end{equation*}
Since $\zeta$ is a limit point of the forward $T$-orbit of $\triv(\alpha)$, it follows that the open neighborhood $\pi^{-1}(\overline{\{\pi(\zeta)\}})$ of $\zeta$ contains $T^b(\triv(\alpha))$ for some integer $b\geq0$.  We conclude that 
\begin{equation*}
\phi^b(\alpha)=\phi^b(\pi(\triv(\alpha)))=\pi(T^b(\triv(\alpha)))\in \overline{\{\pi(\zeta)\}}
\end{equation*} 
and therefore
\begin{equation*}
\phi^{a\ell+b}(\alpha) = (\phi^a)^{\ell}(\phi^b(\alpha))\in\overline{\{\pi(\zeta)\}}\subseteq V
\end{equation*}
for all $\ell\in\NN$, verifying that the set $\{n\in\NN\mid\phi^n(\alpha)\in V\}$ contains the infinite arithmetic progression $a\NN+b$, and completing the proof of Theorem~\ref{MainThmIntro}.  It now remains only to give the proofs of the two claims.

\medskip

{\em Proof of Claim 1:}  Consider an arbitrary continuous function $F:\Xbf\to\RR$.  Then for each $m\geq1$ we have
\begin{equation*}
\int (F\circ T) d\mu - \int Fd\mu = A_m+B_m+C_m 
\end{equation*}
where
\begin{equation*}
\begin{split}
A_m & = \int (F\circ T)d(\mu-\mu_{m}) \\
B_m & = \int Fd(\mu_{m}-\mu) \\
C_m & = \int ((F\circ T)-F)d\mu_{m}.
\end{split} 
\end{equation*}
The weak convergence $\mu_{m}\to\mu$ implies that $A_{m}\to0$ and $B_{m}\to0$, and writing $I_{m}=\{a_m,a_m+1,\dots,b_m\}$, we have
\begin{equation*}
\begin{split}
|C_{m}| & = \bigg|\frac{1}{|I_{m}|}\sum_{n\in I_{m}}(F(T^{n+1}(\triv(\alpha)))-F(T^n(\triv(\alpha))))\bigg| \\
	& = \bigg|\frac{1}{|I_{m}|}(F(T^{b_{m}+1}(\triv(\alpha)))-F(T^{a_{m}}(\triv(\alpha))))\bigg| \\
	& \leq \frac{2}{|I_{m}|}\sup_{x\in\Xbf}|F(x)|.
\end{split} 
\end{equation*}
Since $|I_m|\to+\infty$ we have $C_{m}\to0$, and so $\int (F\circ T) d\mu=\int Fd\mu$, completing the proof that $\mu$ is $T$-invariant.

\medskip

{\em Proof of Claim 2:}  Since the closed set $\triv(V)$ is a subset of the open set $\pi^{-1}(V)$, Urysohn's lemma (\cite{MR1681462} Thm 4.15) implies that there exists a continuous function $G:\Xbf\to[0,1]$ such that $G(x)=0$ for all $x\in\Xbf\setminus\pi^{-1}(V)$ and $G(x)=1$ for all $x\in\triv(V)$.  By the weak convergence $\mu_{m}\to\mu$ and $(\ref{Epsilon})$ we have
\begin{equation}\label{GIneq1}
\begin{split}
\int G d\mu & = \lim_{m\to+\infty}\frac{1}{|I_m|}\sum_{n\in I_m}G(T^n(\triv(\alpha))) \\ 
	& = \lim_{m\to+\infty}\frac{1}{|I_m|}\sum_{n\in I_m}G(\triv(\phi^n(\alpha))) \\
	& \geq  \lim_{m\to+\infty}\frac{|\{n\in\NN\mid \phi^n(\alpha)\in V\}\cap I_m|}{|I_m|} >0.
\end{split} 
\end{equation}
On the other hand $\int G d\mu \leq \int_{\pi^{-1}(V)}d\mu=\mu(\pi^{-1}(V))$, and combining this with $(\ref{GIneq1})$ establishes Claim 2.
\end{proof}

\begin{proof}[Proof of Theorem~\ref{MainCorIntro}]

For the purposes of this proof, let us call a subset of $\NN$ {\em asymptotically periodic} (AP) if it is equal to the union of a finite (possibly empty) collection of infinite arithmetic progressions together with a set of Banach density zero.  Observe that: (i) finite sets are AP; (ii) finite unions of AP sets are AP; (iii) if $S$ is AP, then so is $\{a\ell+b\mid \ell\in S\}$ for each $a\geq1$, $b\geq0$.

Given an affine $k$-variety, a morphism $\phi: X\to X$, a point $\alpha\in X$, and a Zariski-closed subset $V$ of $X$, we must show that $\{n\in\NN\mid\phi^n(\alpha)\in V\}$ is AP.  The proof will use induction on $\dim(V)$; in the zero-dimensional case $V$ is a finite set and the result is trivial.  Suppose that $\dim(V)\geq1$ and assume that the theorem is true for Zariski-closed subsets of $X$ of dimension less than $\dim(V)$.  If $V_1,\dots,V_r$ are the irreducible components of $V$, then 
\begin{equation*}
\{n\in\NN\mid\phi^n(\alpha)\in V\}=\cup_j\{n\in\NN\mid\phi^n(\alpha)\in V_j\},
\end{equation*}
and so we may assume without loss of generality that $V$ is irreducible.

We may assume that there exists an infinite arithmetic progression $a\NN+b$ (for $a\geq1$ and $b\geq0$) such that 
\begin{equation}\label{ContainsProg}
\phi^{n}(\alpha)\in V \text{ for all }n\in a\NN+b,
\end{equation}
because otherwise we are in the situation of Theorem~\ref{MainThmIntro}, which we have already proved.  

For each integer $j\geq0$, define
\begin{equation*}
W_j=\overline{\{\phi^n(\alpha)\mid n\in a\NN+j\}}.
\end{equation*}
In particular, we have $W_b\subseteq V$ by $(\ref{ContainsProg})$ and the fact that $V$ is Zariski-closed.  For each $j\geq0$ we have 
\begin{equation*}
\begin{split}
\phi(W_j) & =\phi(\overline{\{\phi^n(\alpha)\mid n\in a\NN+j\}}) \\
	& \subseteq \overline{\phi(\{\phi^n(\alpha)\mid n\in a\NN+j\})} \\
	& = \overline{\{\phi^n(\alpha)\mid n\in a\NN+j+1\}} \\
	& = W_{j+1}.
\end{split}
\end{equation*}
Thus for each $j\geq0$ the morphism $\phi$ restricts to a map $\phi_j:W_j\to W_{j+1}$.  Moreover, each map $\phi_j$ is dominant, since $W_{j+1}$ is defined to be the Zariski-closure of a set of points in the $\phi$-image of $W_j$.  A consequence is that the sequence $\{\dim(W_j)\}$ is nonincreasing, and therefore must stabilize.  Possibly replacing $b$ with a larger element of its congruence class modulo $m$, we may assume without loss of generality that $\dim(W_j)=\dim(W_b)$ for all $j\geq b$.

Each integer $n\geq b$ is contained in one of the $a$ infinite arithmetic progressions $a\NN+j$ for $b\leq j\leq b+a-1$, and therefore
\begin{equation*}
\{n\in\NN\mid\phi^n(\alpha)\in V\}= S\cup S_b\cup S_{b+1}\cup\dots\cup S_{b+a-1}
\end{equation*}
where $S$ is a subset of $\{0,1,\dots,b-1\}$ and
\begin{equation*}
S_j=\{n\in a\NN+j\mid\phi^n(\alpha)\in V\cap W_j\}.
\end{equation*}
We must show that each set $S_j$ is AP.

\medskip

\underline{Case 1, $\dim(W_b)<\dim(V)$:}  Define $\psi=\phi^a$ and for each $j$ let $\beta=\phi^j(\alpha)$.  Then 
\begin{equation*}
S_j=\{a\ell+j\mid\ell\in\NN\text{ and }\psi^\ell(\beta)\in V\cap W_j\}
\end{equation*}
is AP by the induction hypothesis and the fact that 
\begin{equation*}
\dim(V\cap W_j)\leq\dim(W_j)=\dim(W_b)<\dim(V).
\end{equation*}

\medskip

\underline{Case 2, $\dim(W_b)=\dim(V)$:}  Since $W_b\subseteq V$ and $V$ is irreducible, we must have $W_b=V$.  Since $W_b$ is irreducible, it follows from the existence of the dominant maps $\phi_j:W_j\to W_{j+1}$ that $W_j$ is irreducible for all $j\geq b$.  Given $b\leq j\leq b+a-1$, we must have either $\dim(V\cap W_j)<\dim(V)$ or $\dim(V\cap W_j)=\dim(V)$.  In the former case, $S_j$ is AP by the induction hypothesis and the same argument used in Case 1.  In the latter case, we have
\begin{equation*}
\dim(V\cap W_j)=\dim(V)=\dim(W_b)=\dim(W_j)
\end{equation*}
and we conclude that $W_j=V$, since $W_j$ and $V$ are both irreducible.  It follows from the identity $W_j=V$ that $\phi^n(\alpha)\in V$ for all $n\in a\NN+j$, and we conclude that $S_j=a\NN+j$, which is AP.  
\end{proof}

\end{document}